\documentclass[11pt,a4paper]{article}
\usepackage{latexsym,amsfonts,amsmath,graphics}
\usepackage{xspace}
\usepackage{epsfig}
\usepackage{enumitem}
\usepackage[usenames,dvipsnames,svgnames,table]{xcolor}

\usepackage{makeidx}         % allows index generation
\usepackage{graphicx}        % standard LaTeX graphics tool
\usepackage[bottom]{footmisc}% places footnotes at page bottom

\usepackage{amssymb}
\usepackage{bm}

\usepackage{url}
\usepackage{marvosym}
\usepackage{booktabs}

\newtheorem{theorem}{Theorem}
\newtheorem{lemma}{Lemma}

\newenvironment{proof}{\begin{trivlist}
\item[\hskip\labelsep{\it Proof.}]}{$\hfill\Box$\end{trivlist}}

\newcommand{\abs}[1]{\left\vert#1\right\vert}
\newcommand{\tlambda}{\widetilde{\lambda}}
\newcommand{\naturals}{\mathbb{N}}
\newcommand{\me}{\mathrm{e}}
\newcommand{\SEWT}{\sigma_{\;\textup{EWT}}}
\newcommand{\ALG}{\textup{ALG}\xspace}
\newcommand{\EXP}{\textup{EXP}\xspace}
\newcommand{\NOR}{\textup{NOR}\xspace}
\newcommand{\ABS}{\textup{ABS}\xspace}
\newcommand{\SPT}{\textup{SPT}\xspace}
\newcommand{\PT}{\textup{PT}\xspace}
\newcommand{\QPT}{\textup{QPT}\xspace}
\newcommand{\WT}{\textup{WT}\xspace}
\newcommand{\UWT}{\textup{UWT}\xspace}
\newcommand{\fraku}{\mathfrak{u}}

\newcommand{\calG}{{\mathcal{G}}}
\newcommand{\calH}{{\mathcal{H}}}
\newcommand{\calS}{{\mathcal{S}}}
\newcommand{\C}{{\mathbb{C}}} % complex numbers
\newcommand{\N}{{\mathbb{N}}} % natural numbers {1, 2, ...}
\newcommand{\R}{{\mathbb{R}}} % reals
\definecolor{darkmagenta}{RGB}{139,0,139}

\begin{document}

\title{Exponential tractability of linear tensor product problems}

\author{Fred J. Hickernell\thanks{F.~J.~Hickernell gratefully acknowledges support by the United States National Science Foundation grant DMS-1522687.},
Peter Kritzer\thanks{P.~Kritzer gratefully acknowledges support by the Austrian Science Fund (FWF): Project F5506-N26, which is part
of the Special Research Program ``Quasi-Monte Carlo Methods: Theory and Applications''.}, 
Henryk Wo\'{z}niakowski\thanks{H.~Wo\'{z}niakowski gratefully acknowledges the support of the National Science Centre, Poland, based on the decision 
DEC-2017/25/B/ST1/00945.}}

\maketitle

\begin{abstract}
In this article we consider the approximation of compact   
linear operators defined over tensor product Hilbert   
spaces. Necessary and sufficient conditions on the singular values of the problem 
under which we can or cannot achieve different notions of exponential tractability 
are given in \cite{PPW17}. In this paper, we use the new equivalency 
conditions shown in \cite{KW18} 
to obtain these results in an alternative way. 
As opposed to the algebraic setting, quasi-polynomial tractability is not 
possible for non-trivial cases in the exponential setting. 
\end{abstract}

\section{Introduction and Preliminaries}
\label{sec:intro}

Tractability of multivariate problems is the subject of a considerable number 
of articles and monographs in the field of Information-Based Complexity (IBC).  
For an introduction to IBC, we refer to the book \cite{TWW88}.  For a recent 
overview of the state of the art in tractability studies we refer to the 
trilogy \cite{NW08}--\cite{NW12}. In this article we study tractability in   
the worst case setting for linear tensor product problems and 
for algorithms that use finitely many   
arbitrary continuous linear functionals.    

The information complexity of a compact linear operator   
$S_d:\calH_d \to \calG_d$ is defined as the minimal number, $n(\varepsilon,S_d)$, of such   
linear functionals needed to find an $\varepsilon$-approximation. 
It is natural to ask how the information complexity of a given problem 
depends on both $d$ and $\varepsilon^{-1}$.

In most of the literature on this subject, different notions of tractability are
defined in terms of a relationship between $n(\varepsilon,S_d)$ and some powers of $d$ and
$\max(1,\varepsilon^{-1})$. This is called algebraic (ALG) tractability. 
For a complete overview of a wide range of results 
on algebraic tractability, see \cite{NW08}--\cite{NW12}
On the other hand, a relatively recent stream of work defines different notions
of tractability in terms of a relationship between $n(\varepsilon,S_d)$ and some powers of $d$
and $1+\log \max(1,\varepsilon^{-1})$. Now the complexity of the problem increases only logarithmically
as the error tolerance vanishes. This situation is referred to as exponential
(EXP) tractability, which is the subject of this article. Precise definitions of ALG
and EXP tractabilities are given below.

General compact linear multivariate problems have been studied 
in the recent article \cite{KW18}. Here, we deal with the case of    
tensor product problems for which the singular values of a $d$-variate problem 
are given as products of the singular values of univariate problems. 
Exponential tractability for tensor product problems has been studied in \cite{PPW17}. In this 
paper we re-prove the results of \cite{PPW17} by a different argument via the criteria 
presented in \cite{KW18}.

\bigskip

Consider two Hilbert spaces $\calH_1$ and $\calG_1$ and a 
compact linear solution operator,
$
S_1: \calH_1 \to \calG_1.
$
Let $\N$ denote the set of positive integers,   
and $\N_0 =\N\cup\{0\}$.  For $d\in \N$, let 
$$
\calH_d\,=\,\calH_1\otimes\calH_1\otimes\cdots\otimes\calH_1\quad\mbox{and}\quad
\calG_d\,=\,\calG_1\otimes\calG_1\otimes\cdots\otimes\calG_1
$$
be the $d$-fold tensor products of the spaces $\calH_1$ and $\calG_1$, respectively. 
Furthermore, let $S_d$ be the linear tensor product operator,
$$ 
S_d=S_1\otimes S_1\otimes \cdots \otimes S_1,
$$
on $\calH_d$. 
In this way, obtain a sequence of compact    
linear solution operators    
$$   
\calS\,=\, \{S_d: \calH_d \to \calG_d\}_{d \in \N}.     
$$   
We now consider the problem of 
approximating $\{S_d(f)\}$ for $f$ from the   
unit ball of $\calH_d$ by means of algorithms   
$\{A_{d,n}: \calH_d \to \calG_d\}_{d \in \N,n\in\N_0}$.    
For $n=0$, we set $A_{d,0}:=0$, and for $n\ge 1$, $A_{d,n}(f)$    
depends on $n$ continuous    
linear functionals $L_1(f), L_2(f),\ldots, L_n(f)$, so that   
\begin{equation}\label{eq:algorithm} 
A_{d,n}(f)=\phi_n(L_1(f),L_2(f),\dots,L_n(f))   
\end{equation}   
for some $\phi_n:\C^n\to \calG_d$ or $\phi_n:\R^n\to \calG_d$ 
and $L_j\in \calH_d^*$.
We allow an adaptive choice of $L_1,L_2,\ldots,L_n$ as well as $n$, i.e.,   
$L_j=L_j(\cdot;L_1(f),L_2(f),\dots,L_{j-1}(f))$ and $n$ can be a   
function of the $L_j(f)$,    
see \cite{TWW88} and \cite{NW08} for details.
The error of a given algorithm $A_{d,n}$ is measured in the worst case setting, 
which means that we need to deal with  
$$   
e(A_{d,n})\,=\,\sup_{\substack{f\in\calH_d \\ \|f\|_{\calH_d}\le1}}   
\|S_d(f)-A_{d,n}(f)\|_{\calG_d}.   
$$   
However, to assess the difficulty of the approximation problem, we would not 
only like to study the worst case errors of particular algorithms, but 
consider a more general error measure. To this end, let   
$$   
e(n,S_d)\,=\,\inf_{A_{d,n}}\,e(A_{d,n})   
$$   
denote the $n$th minimal worst case error,    
where the infimum is extended over all admissible algorithms $A_{d,n}$ of the form \eqref{eq:algorithm}. Then    
the information complexity $n(\varepsilon,S_d)$ is   
the minimal number $n$ of continuous linear functionals    
needed to find an algorithm $A_{d,n}$ that approximates   
$S_d$ with error at most $\varepsilon$. More precisely,    
we consider the absolute (\ABS) and normalized (\NOR) error criteria   
in which    
\begin{align*}   
 n(\varepsilon,S_d)&=n_{\ABS}(\varepsilon, S_d)\, =   
\min\{n\,\colon\,e(n,S_d)\le \varepsilon\},\\   
 n(\varepsilon,S_d)&=n_{\NOR}(\varepsilon,S_d)=   
\min\{n\,\colon\,e(n,S_d) \le \varepsilon\,\|S_d\|\}.   
\end{align*}   

It is known from \cite{TWW88} (see also \cite{NW08})   
that the information complexity is fully determined by    
the singular values of $S_d$, which    
are the same as the square roots of the eigenvalues of    
the compact self-adjoint and positive semi-definite   
linear   
operator $W_d=S_d^\ast S_d:\calH_d\rightarrow\calH_d$.    
We denote these eigenvalues by $\lambda_{d,1}, \lambda_{d,2},\ldots$. Then it is known that 
the information complexity can be expressed in terms of the eigenvalues $\lambda_{d,j}$. Indeed, 
\begin{align}   
n_{\ABS}(\varepsilon, S_d)   
&=\min\{n\,\colon\,\lambda_{d,n+1}\le   
\varepsilon^2\},
\label{eq:infcomplabs}\\   
n_{\NOR}(\varepsilon, S_d)&=   
\min\{n\,\colon\,\lambda_{d,n+1}\le \varepsilon^2 \lambda_{d,1}\}. 
\label{eq:infcomplnor}     
\end{align}
Clearly,    
$n_{\ABS}(\varepsilon, S_d)=0$ for $\varepsilon\ge   
\sqrt{\lambda_{d,1}}=\|S_d\|$, and    
$n_{\NOR}(\varepsilon, S_d)=0$ for $\varepsilon\ge1$.   
Therefore for $\ABS$ we can restrict ourselves to $\varepsilon\in(0,\|S_d\|)$,   
whereas for $\NOR$ to $\varepsilon\in(0,1)$.   
Since $\|S_d\|$ can be arbitrarily large, to deal simultaneously    
with $\ABS$ and $\NOR$ we consider $\varepsilon\in(0,\infty)$.    
It is known that $n_{\ABS/\NOR}(\varepsilon,S_d)$ is finite for all   
$\varepsilon>0$ iff $S_d$ is compact, which justifies our assumption   
about the compactness of~$S_d$.   

We now recall that the spaces $\calH_d$ and $\calG_d$ are tensor product
spaces. It is known that the eigenvalues $\lambda_{d,j}$ of $W_d$ are then given as products 
of the eigenvalues $\tlambda_j$ of the operator $W_1=S_1^\ast S_1:\calH_1\rightarrow\calH_1$, i.e.,
\begin{equation}\label{eq:productform}
\lambda_{d,j}=\prod_{\ell=1}^d \tlambda_{j_\ell}.
\end{equation}
Without loss of generality, we assume that the $\tlambda_j$ are ordered, i.e., $\tlambda_1\ge \tlambda_2 \ge \cdots$. 

Although the $\tlambda_j$ are given by \eqref{eq:productform}, the ordering of the $\tlambda_j$ does not easily imply the
ordering of the $\lambda_{d,j}$ since the map $j \in \naturals \mapsto (j_1, \ldots, j_d) \in \naturals^d$
exists but does not have a simple explicit form. 
This makes the tractability analysis challenging.
   
We are ready to define exactly various notions of ALG and EXP tractabilities. To
present them concisely, let
\begin{equation*}
\mathsf{y} \in \{\ABS, \NOR \}, \quad
\mathsf{z} = \begin{cases} \max(1,\varepsilon^{-1}), & \mbox{in the case of ALG}, \\
1+\log\, \max(1,\varepsilon^{-1}), & \mbox{in the case of EXP}.\end{cases}
\end{equation*}
These definitions are as follows.

\bigskip

The problem $\calS$ is \ldots 

\begin{itemize}
 \item strongly polynomially tractable (\SPT) if there are $C, q\ge 0$ such that
$$
 n_{\texttt{y}}(\varepsilon,S_d) \le C \mathsf{z}^q\quad \forall d \in \naturals, \, \varepsilon \in (0,\infty),
$$

\item polynomially tractable (\PT) if there are $C,p,q\ge 0$ such that
$$
 n_{\texttt{y}}(\varepsilon,S_d) \le C d^p \mathsf{z}^q\quad \forall d \in \naturals, \, \varepsilon \in (0,\infty),
$$

\item quasi-polynomially tractable (\QPT) if there are $C,p\ge 0$ such that
$$
 n_{\texttt{y}}(\varepsilon,S_d) \le C\,\exp\bigl(p\,(1+\log\,d)(1+\log\, \mathsf{z})\bigr)
 \quad \forall d \in \naturals, \, \varepsilon \in (0,\infty),
$$
\item $(s,t)$-weakly tractable ($(s,t)$-\WT) if
$$
\lim_{d+\varepsilon^{-1}\to\infty}\   
\frac{\log\,\max(1,n_{\mathsf{y}}(\varepsilon,S_d))}{d^{\,t}+\mathsf{z}^{s}}\,=\,0,
$$

\item uniformly weakly tractable (\UWT) if $(s,t)$-\WT holds for all $s,t>0$.

\end{itemize}

We use the prefix \ALG- with the above tractability notions in the case $\mathsf{z}=\max(1,\varepsilon^{-1})$ and 
\EXP- in the case $\mathsf{z}=1+\log\, \max(1,\varepsilon^{-1})$.

A recent article \cite{KW18} provides  necessary and sufficient conditions 
on the eigenvalues $\lambda_{d,j}$ of $W_d$ for the various tractability 
notions above. For the special case of linear tensor product spaces 
considered here, it is natural 
to ask for conditions on the eigenvalues $\tlambda_j$ of $W_1$ such that we obtain 
the different kinds of exponential tractability. For results on algebraic tractability 
for tensor product spaces, see again \cite{NW08}--\cite{NW12} and the articles cited therein. 
The notion of $(s,t)$-\WT was introduced in \cite{SW15}, and \UWT was introduced in \cite{S13}. 
See also \cite{WW17} and \cite{KW18} for results on $(s,t)$-\WT and \UWT in the algebraic sense.

Finding  
necessary and sufficient conditions on the $\tlambda_j$ for the different kinds of exponential 
tractability turns out to be a technically difficult question. Necessary and sufficient conditions have 
been considered in the paper \cite{PPW17}. Here we re-prove the results in \cite{PPW17}, using 
a completely different technique, namely using criteria that have been shown very recently in the 
paper \cite{KW18}. In some cases, our new technique enables us to obtain the desired results using
shorter and/or less technical arguments than those that were used in \cite{PPW17}. 

\section{Results}
\label{sec:results}
 
In this section we show results on tractability conditions in terms of the eigenvalues of the 
operator $W_1$. 

Note that, if all of the $\tlambda_j$ equal zero, then the operators $S_d$ are all zero, and 
$n_{\ABS/\NOR}(\varepsilon,S_d)=0$ for all $d\ge 1$. Furthermore, 
if only $\tlambda_1>0$ and $\tlambda_2=\tlambda_3=\cdots =0$ (remember that the $\tlambda_j$ are ordered), 
it can be shown that $n_{\ABS/\NOR}(\varepsilon,S_d)\le 1$ for all $d\ge 1$. Hence, the problem is 
interesting only if at least two of the $\tlambda_j$ are positive, which we assume from now 
on. 

Before we state our main result, we state two technical lemmas and a theorem proved elsewhere.  The first lemma is well known.

\begin{lemma} \label{sSumLem}
For any $n \in \N$  and $a_1,a_2,\ldots,a_n \ge 0$ we have:
\begin{itemize}
\item For $s\ge1$,
\begin{equation*}
\frac1{n^{s-1}} (a_1 + \cdots + a_n)^s
\le a_1^s + \cdots + a_n^s \le
(a_1 + \cdots + a_n)^s.
\end{equation*}
\item For $s\le1$
\begin{equation*}
\qquad
\qquad
(a_1 + \cdots + a_n)^s
\le a_1^s + \cdots + a_n^s \le n^{1-s}(a_1 + \cdots + a_n)^s.
\end{equation*}
\end{itemize}
\end{lemma}

\begin{lemma} \label{binomBdLem} For all $k, n \in \N$ with $k < n$, it follows that
\begin{equation*}
      \max\left\{\left(\frac nk\right)^k,  \left(\frac n{n-k}\right)^{n-k} \right\}
      \le \binom{n}{k} \le \min\left\{\left(\frac {\me n}k\right)^k,  \left(\frac {\me n}{n-k}\right)^{n-k} \right\}.
\end{equation*}
\end{lemma}
\begin{proof}
It is easy to see that
                \begin{align*}
                \left(\frac nk\right)^k  = \frac{n}{k}\cdots\frac{n}{k}           
                & \le \frac{n (n-1) \cdots (n - k+1)}{k (k-1) \cdots  1}
                 =\binom{n}{k} = \binom{n}{n-k} \\
                & \le \frac{n^k}{k!} = \left(\frac nk\right)^k \frac{k^k}{k!} \le  \left(\frac {\me n}k\right)^k,
                \end{align*}
and the estimates of the lemma follow.$\hfill\Box$
\end{proof}

\begin{theorem} \cite[Theorem 3]{KW18} \label{thm:general}   
$\calS$ is \EXP-$(s,t)$-\WT-\ABS/\NOR iff   
\begin{equation}\label{eq:EXP-WT}   
\sup_{d\in\N} \SEWT (d,s,t,c)    
 < \infty\quad \forall c>0,   
\end{equation}   
where   
\[   
 \SEWT (d,s,t,c):=\exp (-cd^{\,t})\,\sum_{j=1}^\infty    
\exp \left(-c \left[   
1+   
\log\left(2\,\max\left(1,\frac{{\rm   
        CRI}_d}{\lambda_{d,j}}\right)\right)   
\right]^s \right),   
\]   
where 
$$   
 \mathrm{CRI}_d=\begin{cases}   
                 1 &\mbox{for $\ABS$},\\   
                 \lambda_{d,1} &\mbox{for $\NOR$}.   
                 \end{cases}   
$$  
\end{theorem}

We now state and prove the main result of this article.
\begin{theorem}\label{thm:main}
Let
$$
\tlambda_1\ge\tlambda_2>0.
$$
Consider the conditions
\begin{equation}\label{rateA}
\lim_{n\to\infty}\ \frac{\log\,\tlambda_{n}^{-1}}{(\log\,n)^{1/\min (s,t)}}=\infty,
\end{equation}
\begin{equation}\label{rateB}
\lim_{n\to\infty}\ \frac{\log\,\tlambda_{n}^{-1}}{(\log\,n)^{1/s}}=\infty,
\end{equation}
\begin{equation}\label{rateC}
\lim_{n\to\infty}\ \frac{\log\,\tlambda_{n}^{-1}}{(\log\,n)^{1/\eta}}=\infty, 
\end{equation}
where $\eta$ is given below.

\noindent\EXP-$(s,t)$-\WT-\ABS holds iff one of the following conditions is true:
\begin{itemize}
\item{\textbf{(A.1):}}\, $t>1$, $s>1$, $\tlambda_1>1$, and \eqref{rateA} holds\ \ \mbox{or}
\item{\textbf{(A.2):}}\, $t>1$, $s\ge 1$, $\tlambda_1\le 1$, and \eqref{rateB} holds\ \ \mbox{or}
\item{\textbf{(A.3):}}\, $t>1$, $s < 1$, and \eqref{rateC} holds with $\eta=s(t-1)/(t-s)$\ \ \mbox{or} 
\item{\textbf{(A.4):}}\, $t\le 1$, $s>1$, $\tlambda_1\le1$, $\tlambda_2<1$, and \eqref{rateB} holds.
\end{itemize}
\EXP-$(s,t)$-\WT-\NOR holds iff one of the following conditions is true:
\begin{itemize}
\item{\textbf{(N.1):}}\, $t>1$, $s\ge 1$, and \eqref{rateB} holds\ \ \mbox{or}
\item{\textbf{(N.2):}}\, $t>1$, $s < 1$, and \eqref{rateC} holds with $\eta=s(t-1)/(t-s)$\ \ \mbox{or} 
\item{\textbf{(N.3):}}\, $t\le 1$, $s>1$, $\tlambda_1>\tlambda_2$, and \eqref{rateB} holds.
\end{itemize}
Furthermore, \EXP-\UWT, \EXP-\QPT, \EXP-\PT, and \EXP-\SPT do not hold under any conditions
on $\tlambda_j$, i.e., even for $\tlambda_3\,=\,\tlambda_4\,=\, \cdots\, =\, 0$.

\end{theorem}

\begin{proof}

We know from Theorem \ref{thm:general} that \EXP-$(s,t)$-\WT holds iff

        \begin{equation} \label{EXPWTcond}
        \sup_{d \in \naturals}  \SEWT(d,s,t ,c) < \infty \quad \forall c>0,\\
        \end{equation}
        where
        \begin{align}
        \nonumber
        \lefteqn{\SEWT(d,s,t,c )} \\
        \label{SEWTSingSum}
         & :=  \sum_{j=1}^\infty \exp\left( -c \left \{d^t + \left [ \log(2\me) + 
         \log \left(\max\left(1, \mathrm{CRI}_d/\lambda_{d,j} \right) \right) \right]^s \right \}\right) \\
         \label{SEWTmultisum}
        &=  \sum_{j_1=1}^\infty \cdots   \sum_{j_d=1}^\infty   \exp\left( -c \left \{ d^t + 
        \left [ \log \left(2 \me \max\left(1, \prod_{\ell = 1}^d \mathrm{CRI}/\tlambda_{j_\ell} \right) \right) \right]^s  \right\} \right),
        \end{align}
where
\[
 \mathrm{CRI}_d=\begin{cases}
               1 & \mbox{for \ABS},\\
               \lambda_{d,1} &\mbox{for \NOR},
              \end{cases}
 \quad\quad \mbox{and}\quad \quad             
 \mathrm{CRI}=\begin{cases}
               1 & \mbox{for \ABS},\\
               \tlambda_1 &\mbox{for \NOR}.
              \end{cases}              
\]

We first show the necessity of the conditions on the eigenvalues $\tlambda_1$ and $\tlambda_2$ 
for ABS and NOR, 
and then the necessity of the conditions \eqref{rateA}, \eqref{rateB}, or \eqref{rateC}, 
depending on the different cases. Then, we show the sufficient conditions for all the cases (A.1)--(A.4) and (N.1)--(N.3). 

\bigskip

\noindent NECESSARY CONDITIONS:

\bigskip

\paragraph{\textbf{Case I:} $t\le 1 \ \&  \ \tlambda_1 > 1 \implies$ NO \EXP-$(s,t)$-\WT-\ABS} \label{Case111f}

        Choose the smallest non-negative $r$ such that $\tlambda_2 \ge
        \tlambda_1^{-r}$, and for every $d > r+2$, let $k = \lfloor
        d/(r+2) \rfloor$.
        Then it follows that
        \begin{equation*}
        k \le \frac{d}{r+2}, \qquad d - k - kr \ge d\left[1 -\frac{r+1}{r+2} \right] = \frac{d}{r+2}.
        \end{equation*}
        Focusing on just these eigenvalues of the form
        \begin{equation*}
        \lambda_{d,j} = \tlambda_1^{d-k}  \tlambda_2^{k} =  \tlambda_1^{d-k-kr}  \tlambda_1^{kr}\tlambda_2^{k} \ge
        \tlambda_1^{d-k-kr}  \tlambda_1^{kr} \tlambda_1^{-kr} =\, \tlambda_1^{d-k-kr} \ge \tlambda_1^{d/(r+2)} \,>\, 1,
        \end{equation*}
        $\SEWT(d,s,t,c )$ has the following lower
        bound via \eqref{SEWTSingSum} and Lemma \ref{binomBdLem}.
        \begin{align*}
        \nonumber
        \SEWT(d,s,t,c )
        & \ge \binom{d}{k} \exp\left( -c \left \{d^t + \left [ \log(2\me) + \log \left(\max\left(1, \lambda_{d,j}^{-1} \right) \right) \right]^s \right \}\right) \\
        & \ge \left(\frac{d}{k}\right)^k \exp\left( -c \left \{d^t + \left [ \log(2\me) \right]^s \right \}\right) \\
        & \ge \left(r+2\right)^{d/(r+2) - 1} \exp\left( -c \left \{d^t + \left [ \log(2\me) \right]^s \right \}\right) \\
        & = \frac{\bigl[(r+2)^{1/(r+2)}\bigr]^d}{r+2} \exp\left( -c \left \{d^t + \left [ \log(2\me) \right]^s \right \}\right).
        \end{align*}
        Since $(r+2)^{1/(r+2)} > 1$ and $t \le 1$, then $\SEWT(d,s,t,c ) \to \infty$ for small $c$ as $d \to \infty$, regardless of the value of $s$.
        Hence, we do not have \EXP-$(s,t)$-\WT-\ABS.

\paragraph{\textbf{Case II:} $t\le 1 \ \&  \ \tlambda_1\ge\tlambda_2\ge1 \implies$ NO \EXP-$(s,t)$-\WT-\ABS }
 We have $2^d$ eigenvalues no smaller than $1$. Therefore
        $$
        \SEWT(d,s,t,c )\ge 2^d\exp\left(-c(d^t+[\log(2e)]^s)\right)
        \to \infty
        $$
        for small $c$ independently of $s$.
        Hence, we do not have \EXP-$(s,t)$-\WT-\ABS.

\paragraph{\textbf{Case III:} $t\le 1 \ \&  \ s \le 1 \implies$ NO \EXP-$(s,t)$-\WT-\ABS} \label{Case111e} 

We have $2^d$ eigenvalues no smaller than $\tlambda_2^d$.
We then have
\begin{align*}
      \SEWT(d,s,t,c) & \ge 2^d
            \exp\left(-c\left(d^t +
            \left[\log\left(2\me \max(1,\tlambda_2^{-d})
            \right)\right]^s\right)\right)\\
      & = 
      \exp\left(d \log \, 2 - cd^t -c\left[\log (2\me) +d\log \, \max(1,\tlambda_2^{-1})\right]^s\right).
\end{align*}
Since $s,t\le 1$ and since $c$ can be arbitrarily small,
we see that this latter term is not bounded for $d\rightarrow\infty$.
Hence, we do not have \EXP-$(s,t)$-\WT-\ABS.

\vskip 1pc
        
{}From the analysis of all these cases, we see that
\EXP-$(s,t)$-\WT-\ABS may only hold when $t>1$, or
when $t\le 1<s$, 
$\tlambda_1\le1$, and $\tlambda_2<1$. This completes the proof of the necessary 
conditions on $\tlambda_1$ and $\tlambda_2$ for \ABS.

\vskip 1pc
We turn to the necessary conditions on $\tlambda_1$ and $\tlambda_2$ for \NOR. This corresponds to considering the ratios
$\lambda_{d,1}/\lambda_{d,j}$ which are at least $1$.
We know that \EXP-$(s,t)$-\WT-\NOR holds iff
$\sup_{d \in \naturals}  \SEWT(d,s,t ,c) < \infty$ for all $c>0$,
where
\begin{eqnarray*}
\SEWT(d,s,t,c )&=&
\sum_{j=1}^\infty \exp\left( -c \left \{d^t +
\left [ \log(2\me) + \log
  \left(\frac{\lambda_{d,1}}{\lambda_{d,j}}\right)
\right) \right]^s\right) \\
&=&  \sum_{j_1=1}^\infty \cdots
\sum_{j_d=1}^\infty  \exp\left( -c \left\{ d^t +
\left [ \log \left(2 \me \prod_{\ell = 1}^d
    \frac{\tlambda_1}{\tlambda_{j_\ell}}
 \right) \right]^s \right\}\right).
 \end{eqnarray*}
Hence, it is the same as \ABS if we assume that $\tlambda_1=1$.
Using the previous results on necessary conditions for \ABS with $\tlambda_1=1$ we obtain the results
for the parameters $s$, $t$, $\tlambda_1$, and $\tlambda_2$ for \NOR. 

\bigskip

Next, we show the necessity of the conditions \eqref{rateA}, \eqref{rateB}, or \eqref{rateC}, 
depending on the different cases.

\paragraph{\textbf{Necessity of \eqref{rateA}:}}

The necessity of \eqref{rateA} for the corresponding subcases follows from Items L1 and L2 of Lemma 1 in \cite{PPW17}.
We remark that these are only based on general definitions, and do not require the technical results used 
in the proof of the main theorem of \cite{PPW17}.

\paragraph{\textbf{Necessity of \eqref{rateB}:}}

Assume first that EXP-$(s,t)$-WT-ABS/NOR holds and that the parameters $t,s, \tlambda_1$, and $\tlambda_2$  
are as in Case (A.2), (A.4), (N.1), or (N.3), respectively. We prove
that \eqref{rateB} holds.
Take $d=1$. Then we know that
$$
\lim_{\varepsilon\to0} \frac{\log\ \max(1,n_{\rm
    ABS/NOR}(\varepsilon,S_1))}
{(\log \varepsilon^{-1})^s}=0.
$$
This means that for any (small) positive $\beta$ there is a positive
$\varepsilon_\beta$ such that
$$
\log\,\max\left(1,n_{\rm ABS/NOR}(\varepsilon,S_1)\right)\le
\beta(\log\,\varepsilon^{-1})^s\ \ \
\mbox{for all $\varepsilon\le \varepsilon_\beta$},
$$
and equivalently
$$
\varepsilon^2\le \exp\left(-2/\beta^{1/s}\left(\log\,\max(1,n_{\rm
    ABS/NOR}(\varepsilon,S_1))\right)^{1/s}\right)\ \ \
\mbox{for all $\varepsilon\le \varepsilon_\beta$}.
$$
Let $n=n_{\rm ABS/NOR}(\varepsilon,S_1)$. Since $\tlambda_{n+1}\le
\varepsilon^2{\rm CRI}$,
with ${\rm CRI}=1$ for ABS and ${\rm CRI}=\tlambda_1$ for NOR,
we obtain for $n\ge\max(1,n_{\rm ABS/NOR}(\varepsilon_{\beta},S_1))$,
$$
\log \tlambda_{n+1}^{-1} \ge
\frac2{\beta^{1/s}}\,\left(\log\,n\right)^{1/s}
+ \log\,{\rm CRI}^{-1}.
$$
Since $2/\beta^{1/s}$ can be arbitrarily large, this  yields \eqref{rateA}.

\paragraph{\textbf{Necessity of \eqref{rateC}:}}

The necessity of \eqref{rateC} for the corresponding subcases follows from Items L1 and L2 of Lemma 1 in \cite{PPW17}.
We remark that these are only based on general definitions, and do not require the technical results used 
in the proof of the main theorem of \cite{PPW17}.

\bigskip

\bigskip

\noindent SUFFICIENT CONDITIONS:

\bigskip

For technical reasons, we begin the proof by showing Case (A.2).

\paragraph{\textbf{(A.2):} $t>1, s\ge 1$, $\tlambda_1\le 1$ \& \eqref{rateB} $\implies $ \EXP-$(s,t)$-\WT-\ABS}

Due to the assumption that $\tlambda_1\le 1$, it is clear that $\tlambda_j^{-1}\ge 1$ for all $j$.

We then have
\begin{eqnarray*}
 \sigma_{\rm EWT} (d,s,t,c)&=& \exp(-c\, d^t)\,
 \sum_{j_1=1}^\infty \cdots \sum_{j_d=1}^\infty
 \exp\left(-c\left[\log (2\mathrm{e})+\log\left(\prod_{\ell=1}^d \frac{1}{\tlambda_{j_\ell}}\right)\right]^s\right)\\
 &\le&
 \exp(-c\, d^t)\,
 \sum_{j_1=1}^\infty \cdots \sum_{j_d=1}^\infty
 \exp\left(-c\left[\log\left(\prod_{\ell=1}^d \frac{1}{\tlambda_{j_\ell}}\right)\right]^s\right)\\
  &=&
 \exp(-c\, d^t)\,
 \sum_{j_1=1}^\infty \cdots \sum_{j_d=1}^\infty
 \exp\left(-c\left[ \sum_{\ell=1}^d \log \left(\frac{1}{\tlambda_{j_\ell}}\right)\right]^s\right)\\
  &\le&
 \exp(-c\, d^t)\,
 \sum_{j_1=1}^\infty \cdots \sum_{j_d=1}^\infty
 \exp\left(-c \sum_{\ell=1}^d \left[\log \left(\frac{1}{\tlambda_{j_\ell}}\right)\right]^s\right),
\end{eqnarray*}
where we used $s\ge1$ and Lemma \ref{sSumLem} in the last step. This yields
\[
 \sigma_{\rm EWT} (d,s,t,c)\le
 \exp(-c\, d^t)\,
 \left(\sum_{j=1}^\infty \exp\left(-c\left[\log \left(1 /\tlambda_j\right)\right]^s\right)\right)^d\, .
\]
Since Condition \eqref{rateB} holds, we know that 
\[
 \log \left(1 /\tlambda_j\right)= h_j (\log (j+1))^{1/s},
\]
where $(h_j)_{j\ge 1}$ is a sequence with $\lim_{j\rightarrow\infty} h_j =\infty$, i.e., 
\begin{eqnarray*}
 \sigma_{\rm EWT} (d,s,t,c)&\le&
 \exp(-c\, d^t)\,
 \left(\sum_{j=1}^\infty \exp\left(-c\left[h_j (\log (j+1))^{1/s}\right]^s\right)\right)^d\\
 &=&\exp(-c\, d^t)\,
 \left(\sum_{j=1}^\infty \exp\left(-c\, h_j^s\, \log (j+1)\right)\right)^d\\
&=& \exp(-c\, d^t)\,
 \left(\sum_{j=1}^\infty \exp\left(\log \left(1/(j+1)^{\,c\, h_j^s}\right)\right)\right)^d\\
 &=& \exp(-c\, d^t)\,
 \left(\sum_{j=1}^\infty \left(\frac{1}{j+1}\right)^{c\, h_j^s}\right)^d\\
&=& \exp(-c\, d^t)\, A_c^d,
\end{eqnarray*}
where $A_c=\sum_{j=1}^\infty \left(\frac{1}{j+1}\right)^{c\, h_j^s}$ is 
well defined and independent of $d$ since $c\, h_j^s$ is greater than one 
for sufficiently large $j$, and the series is convergent. Hence,
\[
 \SEWT (d,s,t,c)\le \exp(-c\, d^t)\, \exp \left(d \log A_c\right).
\]
As $t>1$, we obtain EXP-$(s,t)$-WT-NOR.

\bigskip

We now show Case (A.1), and the other cases in the same order as they are stated in the theorem.

\bigskip

\paragraph{\textbf{(A.1):} $t>1, s> 1$, $\tlambda_1 >  1$ \& \eqref{rateA} $\implies $ \EXP-$(s,t)$-\WT-\ABS}

\paragraph{\textbf{Subcase (A.1.1):} $s\le t$}

We define $\beta_j:=\tlambda_j / \tlambda_1$ for $j\ge 1$, and consider the information complexity with
respect to $\beta_j$ instead of $\tlambda_j$. We denote the information complexity with respect to 
the sequence $\boldsymbol{\beta}=(\beta_j)_{j\ge 1}$ by $n_{\ABS}^{(\boldsymbol{\beta})}$, and that with respect to 
the sequence $\boldsymbol{\lambda}=(\tlambda_j)_{j\ge 1}$ by $n_{\ABS}^{(\boldsymbol{\lambda})}$. Then it is 
straightforward to see that 
\[
n_{\ABS}^{(\boldsymbol{\lambda})} (\varepsilon,S_d)=n_{\ABS}^{(\boldsymbol{\beta})}(\varepsilon /\tlambda_1^{d/2}, S_d)
\]
Since $s\le t$, we have, by Lemma \ref{sSumLem},
\begin{eqnarray*}
d^t+\left(d\log (\tlambda_1^{1/2}) + \log (1/\varepsilon)\right)^s &\le& d^t +  2^{s-1}\,d^s\, (\log (1/\varepsilon))^s +
2^{s-1}\,(\log (1/\varepsilon))^s\\
&\le & C_s (d^t + \left(\log (1/\varepsilon))^s\right)
\end{eqnarray*}
for some positive constant $C_s$ depending on $s$, but not on $d$ or $\varepsilon$. 
This implies 
\begin{eqnarray*}
 \frac{\log\, n_{\ABS}^{(\boldsymbol{\lambda})} (\varepsilon,S_d)}{d^t + \left(1+\log\, \max(1,\varepsilon^{-1})\right)^s}
 \le 
 \frac{\log\, n_{\ABS}^{(\boldsymbol{\beta})}(\varepsilon /\tlambda_1^{d/2}, S_d)}
 {\widetilde{C}_s\left(d^t + \left(1+\log\, \max(1,\varepsilon^{-1}\, \tlambda_1^{d/2})\right)^s\right)}
\end{eqnarray*}
for some positive constant $\widetilde{C}_s$. This means that 
\EXP-$(s,t)$-\WT-\ABS holds with respect to $\boldsymbol{\lambda}$ if it holds with respect 
to $\boldsymbol{\beta}$. However, as $\beta_j\le 1$ for all $j\ge 1$, and since in this subcase $\min (s,t)=s$, the result 
follows from case (A.2) above.

\paragraph{\textbf{Subcase (A.1.2):} $s>t$}

Assume that \eqref{rateA} holds with $\min (s,t) =t$. We need to show that 
\[
 \lim_{d+\varepsilon^{-1}\rightarrow\infty}\ \frac{\log\, n_{\ABS}^{(\boldsymbol{\lambda})} (\varepsilon,S_d)}{d^t + \left(1+\log\, \max(1,\varepsilon^{-1})\right)^s}=0.
\]
However, note that 
\[
 \frac{\log\, n_{\ABS}^{(\boldsymbol{\lambda})} (\varepsilon,S_d)}{d^t + \left(1+\log\, \max(1,\varepsilon^{-1})\right)^s} \le 
 \frac{\log\, n_{\ABS}^{(\boldsymbol{\lambda})} (\varepsilon,S_d)}{d^t + \left(1+\log\, \max(1,\varepsilon^{-1})\right)^t},
\]
and that 
\[
 \lim_{d+\varepsilon^{-1}\rightarrow\infty}\ \frac{\log\, n_{\ABS}^{(\boldsymbol{\lambda})} (\varepsilon,S_d)}{d^t + \left(1+\log\, \max(1,\varepsilon^{-1})\right)^t}=0
\]
by Case (A.1.1). This shows the result. 

\bigskip

\paragraph{\textbf{(A.3):} $t>1$, $s< 1$, \& \eqref{rateC} with $\eta=s(t-1)/(t-s)$ $\implies $ \EXP-$(s,t)$-\WT-\ABS}

\paragraph{\textbf{Subcase (A.3.1):} $\tlambda_1\le 1$}

We study the expression
\[
 \SEWT (d,s,t,c)\le 
 \exp(-c\, d^t)\,
 \sum_{j_1=1}^\infty \cdots \sum_{j_d=1}^\infty
 \exp\left(-c\left[\sum_{\ell=1}^d \log\left( 1/\tlambda_{j_\ell}\right)\right]^s\right).
\]
Note that the definition of $\eta$ together with $s<1$ implies that $\eta< s< 1$, and $1/\eta > 1/s > 1$. Note furthermore that 
Condition \eqref{rateC} implies
\begin{equation}\label{eq:sequencehj}
 \log \left(1 /\tlambda_j\right)= h_j (\log (j+1))^{1/\eta} =h_j\, (\log (j+1))^{1/\eta-1/s}\, (\log (j+1))^{1/s},
\end{equation}
where $(h_j)_{j\ge 1}$ is a sequence with $\lim_{j\rightarrow\infty} h_j =\infty$.

Using \eqref{eq:sequencehj} and the second item in Lemma \ref{sSumLem}, we obtain
\begin{eqnarray*}
 \left[\sum_{\ell=1}^d \log\left( 1/\tlambda_{j_\ell}\right)\right]^s 
 &=& \left[\sum_{\ell=1}^d h_{j_\ell}\, (\log (j_\ell+1))^{1/\eta-1/s}\, (\log (j_\ell+1))^{1/s}\right]^s \\
 &\ge & d^{s-1}\,\sum_{\ell=1}^d (\log (j_\ell +1))^{(s-\eta)/\eta}\, h_{j_\ell}^s\, \log (j_\ell +1).\\ 
\end{eqnarray*}

Consequently,
\begin{eqnarray*}
 \SEWT (d,s,t,c)&\le& 
 \exp(-c\, d^t)\\
 &&\times \sum_{j_1=1}^\infty \cdots \sum_{j_d=1}^\infty
 \exp\left(-c\, d^{s-1}\,\sum_{\ell=1}^d (\log (j_\ell +1))^{(s-\eta)/\eta}\, h_{j_\ell}^s\, \log (j_\ell +1) \right)\\
&=& \exp(-c\, d^t)\, \left(\sum_{j=1}^\infty \exp\left(-c\, d^{s-1}\,(\log (j+1))^{(s-\eta)/\eta}\, h_j^s\, \log (j+1)\right)\right)^d. 
\end{eqnarray*}
Note that 
\[
 d^{s-1}\,(\log (j+1))^{(s-\eta)/\eta}\ge (c/2)^{\eta / (s-\eta)}
\]
if and only if 
\[
 j\ge \left\lceil\exp\left((c/2)\,d^{(1-s) \eta /(s-\eta)}\right)-1\right\rceil=:J_0=J_0 (c,d,s,\eta).
\]
This implies
\[
 \SEWT (d,s,t,c)\le
\exp(-c\, d^t)\, \left(J_0+\sum_{j=J_0}^\infty \exp\left(-c^{s/(s-\eta)}2^{s/(\eta-s)}\, h_j^s\, \log (j+1)\right)\right)^d.
\]

In the same way as in case (A.2), we conclude that there exists a positive constant $A_c$ such that 
\begin{eqnarray*}
 \SEWT (d,s,t,c)&\le& \exp(-c\, d^t)\, (J_0 + A_c)^d\\
 &=&\exp(-c\, d^t)\, \left(\exp\left((c/2)\,d^{(1-s) \eta /(s-\eta)}\right) + A_c\right)^d .
\end{eqnarray*}
Since $A_c$ is independent of $d$, for sufficiently large $d$, 
\begin{eqnarray*}
\SEWT (d,s,t,c)&\le&\exp(-c\, d^t)\, \left(2\exp\left((c/2)\,d^{(1-s) \eta /(s-\eta)}\right)\right)^d\\
&=&\exp(-c\, d^t)\, 2^d \,\exp\left((c/2)\,d^{(1-s) \eta /(s-\eta)+1}\right).
\end{eqnarray*}
It is easily checked that $(1-s) \eta /(s-\eta)+1=t$, so we obtain
\[
 \SEWT (d,s,t,c)\le \exp(-c\, d^t)\, 2^d \exp ( (c/2)\, d^t)=\exp(-(c/2)\, d^t)\, \exp(d\log 2).
\]
As $t>1$, we obtain \EXP-$(s,t)$-\WT-\ABS.

\paragraph{\textbf{Subcase (A.3.2):} $\tlambda_1> 1$}

We again define $\beta_j:=\tlambda_j / \tlambda_1$ for $j\ge 1$, and consider the information complexity with
respect to $\beta_j$ instead of $\tlambda_j$. Then, as in Case (A.1.1),
\[
n_{\ABS}^{(\boldsymbol{\lambda})} (\varepsilon,S_d)=n_{\ABS}^{(\boldsymbol{\beta})}(\varepsilon /\tlambda_1^{d/2}, S_d)
\]
Since $t>1$ and $s<1$, we have
\[d^t+\left(d\log (\tlambda_1^{1/2}) + \log (1/\varepsilon)\right)^s \approx d^t +(\log (1/\varepsilon))^s,\]
and this implies that \EXP-$(s,t)$-\WT-\ABS holds with respect to $\boldsymbol{\beta}$ if and only if it holds with respect 
to $\boldsymbol{\lambda}$. However, as $\beta_j\le 1$ for all $j\ge 1$, the result now follows from Case (A.3.1), similar to 
Case (A.1.1).

\paragraph{\textbf{(A.4):} $t\le 1$, $s>1$, $\tlambda_1\le1$, $\tlambda_2<1$ \& \eqref{rateB} $\implies $ \EXP-$(s,t)$-\WT-\ABS}

Due to the assumption that $\tlambda_1\le 1$, we again have $\tlambda_j^{-1}\ge 1$ for all $j$, such that
\begin{eqnarray*}
 \sigma_{\rm EWT} (d,s,t,c)
 &\le&
 \exp(-c\, d^t)\,
 \sum_{j_1=1}^\infty \cdots \sum_{j_d=1}^\infty
 \exp\left(-c\left[\log\left(\prod_{\ell=1}^d \frac{1}{\tlambda_{j_\ell}}\right)\right]^s\right).
\end{eqnarray*}
Note that the approach taken in Case (A.2) does not work now since $t\le 1$. 

In the following we write $[d]$ for the index set $\{1,2,\ldots,d\}$ and, for $\fraku\subseteq [d]$, 
$\overline{\fraku}=[d]\setminus\fraku$. Due to \eqref{rateB}, we can find $m\in\naturals$ such that
\begin{equation}\label{eq:choicemA4}
 \tlambda_{j}^{-1}\ge \tlambda_2^{-1} \exp\left((\log j)^{1/s}\left(2/c\right)^{1/s}\right)\quad 
\forall j\ge m+1.
\end{equation}
Now we study
\begin{eqnarray}\label{eq:ubound}
\lefteqn{\sum_{j_1=1}^\infty \cdots \sum_{j_d=1}^\infty
 \exp\left(-c\left[\log\left(\prod_{\ell=1}^d \frac{1}{\tlambda_{j_\ell}}\right)\right]^s\right)=}\nonumber\\
&=&\sum_{\substack{\fraku \subseteq [d]\\ 
\fraku=\{v_1,\ldots,v_{\abs{\fraku}}\}\\ 
\overline{\fraku}=\{w_1,\ldots,w_{d-\abs{\fraku}}\}}}
\sum_{j_{v_1}=1}^m\cdots \sum_{j_{v_{\abs{\fraku}}}=1}^m \ \ 
\sum_{j_{w_1}=m+1}^\infty \cdots \sum_{j_{w_{d-\abs{\fraku}}}=m+1}^\infty 
\exp\left(-c\left[\log\left(\prod_{\ell=1}^d \frac{1}{\tlambda_{j_\ell}}\right)\right]^s\right).\nonumber\\
\end{eqnarray}
Since $s>1$, we have, for fixed $\fraku\subseteq [d]$, 
\[
 \left[\log\left(\prod_{\ell=1}^d \frac{1}{\tlambda_{j_\ell}}\right)\right]^s\;\ge\;
 \left[\log\left(\prod_{\ell\in\fraku} \frac{1}{\tlambda_{j_\ell}}\right)\right]^s +
 \left[\log\left(\prod_{\ell\in\overline{\fraku}} \frac{1}{\tlambda_{j_\ell}}\right)\right]^s,
\]
so the expression in \eqref{eq:ubound} is bounded by 
\begin{multline*}
\sum_{\substack{\fraku \subseteq [d]\\ 
\fraku=\{v_1,\ldots,v_{\abs{\fraku}}\}\\ 
\overline{\fraku}=\{w_1,\ldots,w_{d-\abs{\fraku}}\}}}
\sum_{j_{v_1}=1}^m\cdots \sum_{j_{v_{\abs{\fraku}}}=1}^m 
\exp\left(-c
\left[\log\left(\prod_{\ell\in\fraku} \tlambda_{j_\ell}^{-1}\right)\right]^s \right) \\
\times\sum_{j_{w_1}=m+1}^\infty \cdots \sum_{j_{w_{d-\abs{\fraku}}}=m+1}^\infty 
\exp\left(-c\left[\log\left(\prod_{\ell\in\overline{\fraku}} \tlambda_{j_\ell}^{-1}\right)\right]^s\right).
\end{multline*}
We first study
\[
A_{\fraku}:=\sum_{j_{v_1}=1}^m\cdots \sum_{j_{v_{\abs{\fraku}}}=1}^m 
\exp\left(-c\left[\log\left(\prod_{\ell\in\fraku} \frac{1}{\tlambda_{j_\ell}}\right)\right]^s \right).
\]
There are a total of $m^d$ terms of the form $\prod_{\ell \in\fraku} \tlambda_{j_\ell}^{-1}$ for 
$j_{v_1}, \ldots, j_{v_{\abs{\fraku}}} \in \{1, \ldots, m\}$ 
in $A_{\fraku}$.  For $k = 0, \ldots, \abs{\fraku}$ there are $(m-1)^k\binom{\abs{\fraku}}{k}$ 
of these terms containing the factor of $\tlambda_1^{-1}$ exactly $\abs{\fraku}-k$ times.  
Such terms are bounded below by $\tlambda_2^{-k}$, so we obtain
$$
A_{\fraku} \le \sum_{k=0}^{\abs{\fraku}}\binom{\abs{\fraku}}{k}
\,(m-1)^k\,\exp\left(-c\,\bigl[\log(\tlambda_{2}^{-k})\bigr]^s\right).
$$
We bound
$\binom{\abs{\fraku}}{k}$ by $(\me\, \abs{\fraku}/k)^k$ due to Lemma \ref{binomBdLem}. Hence, we have
\begin{align*}
A_{\fraku} & \le 1 + \abs{\fraku}\,
\max_{k=1,\dots,\abs{\fraku}}\,\exp(f(k)) 
 \le 1 +  \exp\bigl(\log(\abs{\fraku}) + f(k_{\max})\bigr),
\end{align*}
where
\begin{align*}
f(k)&=k+k\log(\abs{\fraku}/k)+k\log(m-1)-ck^s\bigl[\log(\tlambda_2^{-1})\bigr]^s, \\
f'(k)&= \log(\abs{\fraku}/k)+ \log(m-1)- c s k^{s-1}\bigl[\log(\tlambda_2^{-1})\bigr] , \\
f(k_{\max})& = \max_{k\in[1,\abs{\fraku}]}f(k)  \ge \max_{k=1,\dots,\abs{\fraku}}f(k).
\end{align*}
For $\abs{\fraku}$ large enough, we have
\begin{align*}
f'(1)&= \log(\abs{\fraku}) +\log(m-1)- c s(\log(\tlambda_2^{-1}))^s > 0,\\
f'(d)&= \log(m-1)-c s \abs{\fraku}^{s-1}(\log(\tlambda_2^{-1}))^s < 0,
\end{align*}
hence,
the maximum occurs in the interior.  By setting the $f'(k)=0$, we obtain
\begin{align*}
0  & = \log(\abs{\fraku}/k_{\rm max}) + \log(m-1)-
csk_{\rm max}^{s-1}(\log(\tlambda_2^{-1}))^s, \\
f(k_{\rm max})&=k_{\rm max}+k_{\rm max}\log(\abs{\fraku}/k_{\rm max})
+k_{\rm max}\log(m-1)-ck_{\rm
        max}^{s}\left(\log(\tlambda_2^{-1})\right)^s\\
&= k_{\rm max}+c\,(s-1)\,k_{\rm
        max}^{s}\left(\log(\tlambda_2^{-1})\right)^s.
\end{align*}

The nonlinear equation defining $k_{\max}$ above implies that
\begin{align*}
k_{\rm max} &=
\mathcal{O}\left(\left(\log(\abs{\fraku})\right)^{1/(s-1)}\right),\quad\mbox{and}\quad
f(k_{\rm max})=
\mathcal{O}\left( \left(\log(\abs{\fraku})\right)^{s/(s-1)}\right).
\end{align*}
Consequently,
\[
 A_{\fraku}\le \exp\left(\log(\abs{\fraku}) + \mathcal{O}\left( \left(\log(\abs{\fraku})\right)^{s/(s-1)}\right)\right).
\]

Due to the choice of $m$ in \eqref{eq:choicemA4}, we obtain
\begin{eqnarray*}
\left[\log\left(\prod_{\ell\in\overline{\fraku}} \tlambda_{j_\ell}^{-1}\right)\right]^s
&\ge& \left[\log\left(\prod_{\ell\in\overline{\fraku}} \tlambda_{2}^{-1}\right)
 + \log\left(\prod_{\ell\in\overline{\fraku}} \exp\left((\log j_\ell)^{1/s}\left(2/c\right)^{1/s}\right)\right)\right]^s\\
&=& \left[\abs{\overline{\fraku}}\log\left(\tlambda_{2}^{-1}\right)
 + \sum_{\ell\in\overline{\fraku}} (\log j_\ell)^{1/s}\left(2/c\right)^{1/s}\right]^s\\
&\ge& \abs{\overline{\fraku}}^s \left[\log\left(\tlambda_{2}^{-1}\right)\right]^s + 
 \sum_{\ell\in\overline{\fraku}}\frac{1}{c} \log \left(j_\ell^2\right).
\end{eqnarray*}
Consequently,
\begin{eqnarray*}
B_{\fraku}&:=&\sum_{j_{w_1}=m+1}^\infty \cdots \sum_{j_{w_{d-\abs{\fraku}}}=m+1}^\infty 
\exp\left(-c\left[\log\left(\prod_{\ell\in\overline{\fraku}} \tlambda_{j_\ell}^{-1}\right)\right]^s\right)\\
&\le&\exp\left(-c \abs{\overline{\fraku}}^s \left[\log\left(\tlambda_{2}^{-1}\right)\right]^s\right)
\sum_{j_{w_1}=m+1}^\infty \cdots \sum_{j_{w_{d-\abs{\fraku}}}=m+1}^\infty
\exp\left(-c\sum_{\ell\in\overline{\fraku}}\frac{1}{c} \log \left(j_\ell^2\right)\right)\\
&=&\exp\left(-c \abs{\overline{\fraku}}^s \left[\log\left(\tlambda_{2}^{-1}\right)\right]^s\right)
\sum_{j_{w_1}=m+1}^\infty \frac{1}{j_{w_1}^{\; 2}}\cdots \sum_{j_{w_{d-\abs{\fraku}}}=m+1}^\infty\frac{1}{j_{w_d-\abs{\fraku}}^{\; 2}}\\
&\le& \exp\left(-c \abs{\overline{\fraku}}^s \left[\log\left(\tlambda_{2}^{-1}\right)\right]^s\right) 
\left(\zeta (2)\right)^{\abs{\overline{\fraku}}}
\end{eqnarray*}

In total, we obtain
\begin{eqnarray*}
\lefteqn{\sum_{j_1=1}^\infty \cdots \sum_{j_d=1}^\infty
 \exp\left(-c\left[\log\left(\prod_{\ell=1}^d \frac{1}{\tlambda_{j_\ell}}\right)\right]^s\right)\le}\\
&\le& \sum_{\fraku \subseteq [d]} 
\exp\left(-c \abs{\overline{\fraku}}^s \left[\log\left(\tlambda_{2}^{-1}\right)\right]^s+ 
\abs{\overline{\fraku}}\log (\zeta (2))+\log(\abs{\fraku}) + \mathcal{O}\left( \left(\log(\abs{\fraku})\right)^{s/(s-1)}\right)\right)\\ 
&\le& 
\exp\left(\log(d) + \mathcal{O}\left( \left(\log(d)\right)^{s/(s-1)}\right)\right)
\sum_{\overline{\fraku} \subseteq [d]} 
\exp\left(-c \abs{\overline{\fraku}}^s \left[\log\left(\tlambda_{2}^{-1}\right)\right]^s+ 
\abs{\overline{\fraku}}\log (\zeta (2))\right),
\end{eqnarray*}
where we used $\overline{\fraku}=[d]\setminus\fraku$ in the last step. 

Similarly as in the analysis of $A_{\fraku}$, we see that the sum in the latter expression is bounded by 
\[
 \exp\left(\log (d) + \mathcal{O}\left( \left(\log(d)\right)^{s/(s-1)}\right)\right).
\]
This term grows slower with $d$ than $\exp (-cd^t)$, so we obtain \EXP-$(s,t)$-\WT-\ABS, as desired. 

\paragraph{\textbf{(N.1):} $t>1, s\ge 1$ \& \eqref{rateB} $\implies $ \EXP-$(s,t)$-\WT-\NOR}

Note that in this case we have 
\[
\sigma_{\rm EWT} (d,s,t,c)= \exp(-c\, d^t)\,
 \sum_{j_1=1}^\infty \cdots \sum_{j_d=1}^\infty
 \exp\left(-c\left[\log (2\mathrm{e})+\log\left(\prod_{\ell=1}^d \frac{\tlambda_1}{\tlambda_{j_\ell}}\right)\right]^s\right),
\]
since $\tlambda_1 / \tlambda_j \ge 1$ for all $j$. The rest of the argument 
is analogous to that in Case (A.2).

\paragraph{\textbf{(N.2):}  $t>1$, $s< 1$, \& \eqref{rateC} with $\eta=s(t-1)/(t-s)$ $\implies $ \EXP-$(s,t)$-\WT-\NOR}

This case can be treated in a similar way as Case (A.3), Subcase (A.3.1).

\paragraph{\textbf{(N.3):} $t\le 1$, $s>1$, $\tlambda_1>\tlambda_2$, \& \eqref{rateB} $\implies $ \EXP-$(s,t)$-\WT-\NOR}

This case can be treated in a similar way as Case (A.4). 

\bigskip

\bigskip

Regarding all other tractability notions,
we know from above that we do not have \EXP-$(s,t)$-\WT when $t\le 1$ and $s\le 1$. Since \EXP-$(s,t)$-\WT is a weaker tractability
notion than all other tractability notions considered here, we cannot have any other stronger kind of tractability.

This completes the proof of Theorem \ref{thm:main}.

\end{proof}

\section{Acknowledgement}
The authors thank the MATRIX institute in Creswick, VIC, Australia, and its staff for supporting their 
stay during the program ``On the Frontiers of High-Dimensional Computation'' in June 2018. Furthermore, 
the authors thank the RICAM Special Semester Program 2018, during which parts of the paper were written.

\begin{small}
\noindent\textbf{Authors' addresses:}
\\ \\
\noindent Fred J. Hickernell, 
\\
Center for Interdisciplinary Scientific Computation,\\ Illinois Institute of Technology, \\
           Pritzker Science Center (LS) 106A, 3105 Dearborn Street, Chicago, IL 60616, USA\\
 \\
\noindent Peter Kritzer, 
\\
Johann Radon Institute for Computational and Applied Mathematics (RICAM),\\
Austrian Academy of Sciences,\\ Altenbergerstr.~69, 4040 Linz, Austria\\
 \\
\noindent Henryk Wo\'{z}niakowski,
\\
Department of Computer Science,\\ Columbia University,\\ New York, NY, USA\\
Institute of Applied Mathematics,\\ University of Warsaw,\\
ul. Banacha 2, 02-097 Warszawa, Poland
 \\

\noindent \textbf{E-mail:} \\
\texttt{hickernell@iit.edu} \\
\texttt{peter.kritzer@oeaw.ac.at}\\
\texttt{henryk@cs.columbia.edu}\\

\end{small}

\end{document}